\documentclass[11pt]{amsart}%
\usepackage{amsmath}
\usepackage{amsfonts}
\usepackage{amssymb}
\usepackage{graphicx}
\usepackage{MnSymbol}
\usepackage[margin=1 in]{geometry}
\usepackage{hyperref}
\usepackage{setspace}%
\providecommand{\U}[1]{\protect\rule{.1in}{.1in}}
\providecommand{\U}[1]{\protect\rule{.1in}{.1in}}
\providecommand{\U}[1]{\protect\rule{.1in}{.1in}}
\providecommand{\U}[1]{\protect\rule{.1in}{.1in}}
\providecommand{\U}[1]{\protect\rule{.1in}{.1in}}
\providecommand{\U}[1]{\protect\rule{.1in}{.1in}}
\providecommand{\U}[1]{\protect\rule{.1in}{.1in}}
\providecommand{\U}[1]{\protect\rule{.1in}{.1in}}
\providecommand{\U}[1]{\protect\rule{.1in}{.1in}}
\providecommand{\U}[1]{\protect\rule{.1in}{.1in}}
\providecommand{\U}[1]{\protect\rule{.1in}{.1in}}
\providecommand{\U}[1]{\protect\rule{.1in}{.1in}}
\providecommand{\U}[1]{\protect\rule{.1in}{.1in}}
\providecommand{\U}[1]{\protect\rule{.1in}{.1in}}
\newtheorem{theorem}{Theorem}
\newtheorem*{theorem*}{Theorem}

\newtheorem{lemma}[theorem]{Lemma}

\newtheorem{problem}[theorem]{Problem}
\newtheorem{proposition}[theorem]{Proposition}
\newtheorem*{proposition*}{Proposition}

\begin{document}
\date{\today}
\title[Good Pair of Lattices]{Pairs of Full-Rank Lattices With Parallelepiped-Shaped Common Fundamental Domains}
\author[H. Burgiel]{Heidi Burgiel}
\address{Dept.\ of Mathematics \& Computer Science\\
Bridgewater State University\\
Bridgewater, MA 02325 U.S.A.\\
 }
\email{vignon.oussa@bridgew.edu}
\author[V. Oussa]{Vignon Oussa}
\address{Dept.\ of Mathematics \\
Bridgewater State University\\
Bridgewater, MA 02325 U.S.A.\\
 }
\maketitle

\begin{abstract}
We provide a complete characterization of pairs of full-rank lattices in
$\mathbb{R}^{d}$ which admit common connected fundamental domains of the type
$N\left[  0,1\right)  ^{d}$ where $N$ is an invertible matrix of order $d.$
Using our characterization, we construct several pairs of lattices of the type
$\left(  M\mathbb{Z}^{d},\mathbb{Z}^{d}\right)  $ which admit a common
fundamental domain of the type $N\left[  0,1\right)  ^{d}.$ Moreover, we show
that for $d=2,$ there exists an uncountable family of pairs of lattices of the
same volume which do not admit a common connected fundamental domain of the
type $N\left[  0,1\right)  ^{2}.$

\end{abstract}


\section{Introduction}

Let $\left\{  e_{k}:1\leq k\leq d\right\}  $ be a basis for the vector space
$\mathbb{R}^{d}.$ A full-rank lattice $\Gamma$ is a discrete subgroup of
$\mathbb{R}^{d}$ which is generated by the set $\left\{  e_{k}:1\leq k\leq
d\right\}  .$ The number of generators of the lattice is called the rank of
the lattice, and the set $\left\{  e_{k}:1\leq k\leq d\right\}  $ is called a
basis for the lattice. Adopting the convention that vectors in $\mathbb{R}%
^{d}$ are written as $d\times1$ matrices, it is convenient to describe the
lattice $\Gamma$ as $\Gamma=M\mathbb{Z}^{d}=\left\{  Mk:k\in\mathbb{Z}%
^{d}\right\}  $ where the $j^{th}$ column of the matrix $M$ corresponds to the
vector $e_{j}.$ For a full-rank lattice $M\mathbb{Z}^{d},$ the positive number
$\left\vert \det M\right\vert $ which is equal to the volume of the
parallelepiped $M\left[  0,1\right)  ^{d}$ is conveniently called the volume
of the lattice $M\mathbb{Z}^{d}.$

Let $E$ be a Lebesgue measurable subset of $\mathbb{R}^{d}.$ We say that $E$
packs $\mathbb{R}^{d}$ by $\Gamma$ if and only if for any $\lambda,\gamma
\in\Gamma,\gamma\neq\lambda,$ $\left(  E+\lambda\right)  \cap\left(
E+\gamma\right)  =\emptyset.$ Moreover, we say that $E$ is a measurable
fundamental domain of $\Gamma$ if and only if for any $\lambda,\gamma\in
\Gamma,\gamma\neq\lambda,$ $\left(  E+\lambda\right)  \cap\left(
E+\gamma\right)  =\emptyset$ and $\cup_{\gamma\in\Gamma}\left(  E+\gamma
\right)  =\mathbb{R}^{d}.$ It is worth noticing that if $E$ packs
$\mathbb{R}^{d}$ by $\Gamma$ and if the Lebesgue measure of $E$ is equal to
the volume of $\Gamma$ then $E$ is a fundamental domain of $\Gamma.$

According to a remarkable result of Deguang and Wang (Theorem $1.1$ of
\cite{Han Yang Wang}), it is known that two full-rank lattices in
$\mathbb{R}^{d}$ of the same volume have a common fundamental domain. This
result has profound applications in time-frequency analysis
\cite{Heil,Pfander, Grog}. In \cite{Han Yang Wang}, the authors provide a
general procedure for constructing a fundamental domain for any given pair of
lattices of the same volume. However, it is often the case that the
fundamental domains obtained in \cite{Han Yang Wang} are disconnected,
unbounded and difficult to describe. It is therefore natural to ask if it is
possible to characterize pairs of lattices which admit `simple' common
fundamental domains.

Let us be more precise about what we mean by a `simple' fundamental domain for
a lattice. Let $\Gamma_{1},$ and $\Gamma_{2}$ be two full-rank lattices of the
same volume. We say that the pair $\left(  \Gamma_{1},\Gamma_{2}\right)  $ is
a \textbf{good pair of lattices} if and only if there exists an invertible
matrix $N$ of order $d$ such that the parallelepiped $N\left[  0,1\right)
^{d}$ is a common fundamental domain for $\Gamma_{1},\Gamma_{2}.$ Clearly,
such a fundamental domain is a simple set in the sense that it is connected,
star-shaped, convex and is easily described. Although the investigation of
good pairs of lattices is an interesting problem on its own right, it is also
worth noting that common fundamental domains for pairs of lattices which are
bounded and star-shaped are of central importance in the construction of
smooth frames which are compactly supported \cite{Pfander}.

\subsection{Short overview of the paper}

Our main objective in this paper is to provide solutions to the following problems:

\begin{problem}
\label{One}Is it possible to obtain a simple characterization of good pair of lattices?
\end{problem}

\begin{problem}
\label{Two} For which unimodular matrices $M$ is $\left(  M\mathbb{Z}%
^{d},\mathbb{Z}^{d}\right)  $ a good pair of lattices (or not)?
\end{problem}

On one hand, we are able to address Problem \ref{One} in a way that we judge
is satisfactory. On the other, while we are able to describe several
non-trivial families of good pairs of lattices of the type $\left(  M
\mathbb{Z}^{d},\mathbb{Z}^{d}\right)  $, to the best of our knowledge Problem
\ref{Two} is still open.

Here is a summary of the results obtained in this paper:

\begin{itemize}
\item We present a simple yet powerful characterization of good pairs of
lattices in Proposition \ref{Main}, and we describe various properties
(Proposition \ref{changebasis}) of good pairs of lattices.

\item Addressing Problem \ref{Two}, in Proposition \ref{unipotent}, and
Proposition \ref{order3} we construct several non-trivial families of good
pairs of lattices of the type $\left(  M\mathbb{Z}^{d},\mathbb{Z}^{d}\right)
$ in any given dimension. Moreover, in Proposition \ref{notgood} we establish
the existence of an uncountable collection of pairs of lattices in dimension 2
which have the same volume and are not good pairs.

\item We provide methods that can be exploited to construct good pairs of
lattices in higher dimensions from good pairs of lattices in lower dimensions
(Proposition \ref{tensor}.)
\end{itemize}

Among several results obtained in this work, here are the main ones.

\begin{proposition}
\label{Main} Let $\mathbf{0}$ be the zero vector in $\mathbb{R}^{d}.$ Let
$\Gamma_{1}=M_{1}\mathbb{Z}^{d}$ and $\Gamma_{2}=M_{2}\mathbb{Z}^{d}$ be two
full-rank lattices of the same volume. $\left(  \Gamma_{1},\Gamma_{2}\right)
$ is a good pair of lattices if and only if there exists a unimodular matrix
$N$ ($\left\vert \det N\right\vert =1$) such that $N\left(  -1,1\right)
^{d}\cap\left(  M_{1}^{-1}M_{2}\right)  \mathbb{Z}^{d}=\left\{  \mathbf{0}%
\right\}  $ and $N\left(  -1,1\right)  ^{d}\cap\mathbb{Z}^{d}=\left\{
\mathbf{0}\right\}  .$
\end{proposition}

Notice that for any given invertible matrix $M$ of order $d,$ the zero vector
is always an element of the set $\left(  \left(  M_{1}^{-1}M_{2}\right)
\mathbb{Z}^{d}\cup\mathbb{Z}^{d}\right)  \cap M\left(  -1,1\right)  ^{d}.$
Thus, $\left(  M_{1}\mathbb{Z}^{d},M_{2}\mathbb{Z}^{d}\right)  $ is a good
pair of full-rank lattices if and only if there exists a matrix $N$ of order
$d$ such that $\left\vert \det N\right\vert =1$ and the set $\left(  \left(
\left(  M_{1}^{-1}M_{2}\right)  \mathbb{Z}^{d}\right)  \cup\mathbb{Z}%
^{d}\right)  \cap\left(  N\left(  -1,1\right)  ^{d}\right)  $ is a singleton.
We also observe that the condition described in Proposition \ref{Main} is
easily checked (especially in lower dimensional vector spaces) and will be
exploited to derive other results. Additionally, we would like to point out
that since the volume of the set $N\left(  -1,1 \right)  ^{d}$ must be equal
to $2^{d},$ according to a famous theorem of Minkowski (Theorem $2,$
\cite{Til}) the closure of the set $N\left(  -1,1 \right)  ^{d}$ must contain
points of the lattices $\left(  M_{1}^{-1}M_{2}\right)  \mathbb{Z}^{d},
\mathbb{Z}^{d}$ other than the zero vector. Thus, $\left(  M_{1}\mathbb{Z}%
^{d},M_{2}\mathbb{Z}^{d}\right)  $ is a good pair of lattices if and only if
there exists a unimodular matrix $N$ such that the only nonzero elements of
$\left(  M_{1}^{-1}M_{2}\right)  \mathbb{Z}^{d}$ and $\mathbb{Z}^{d}$ which
belong to the closure of the set $N\left(  -1,1\right)  ^{d}$ are on the
boundary of the $N\left(  -1,1\right)  ^{d}$.

\begin{proposition}
\label{unipotent}Let $M$ be a triangular matrix of order $d$ with ones on the
diagonal and let $P,Q$ be unimodular integral matrices of order $d.$ Then
$\left(  PMQ\mathbb{Z}^{d},\mathbb{Z}^{d}\right)  $ is a good pair of lattices
with common fundamental domain $PM\left[  0,1\right)  ^{d}.$
\end{proposition}

Put $\mathbf{p}=\left(  p_{1},\cdots,p_{d-1}\right)  \in\mathbb{R}^{d-1}$ and
define the matrix-valued functions $\mathbf{p}\mapsto M(\mathbf{p})$ and
$\mathbf{p}\mapsto N(\mathbf{p})$ as follows:
\[
M\left(  \mathbf{p}\right)  =\left[
\begin{array}
[c]{cccc}%
p_{1} & 1 &  & \\
& \ddots & \ddots & \\
&  & p_{d-1} & 1\\
&  &  & {\prod\limits_{k=1}^{d-1}}\frac{1}{p_{k}}%
\end{array}
\right]  ,\text{ }N\left(  \mathbf{p}\right)  =\left[
\begin{array}
[c]{ccccc}
&  &  &  & 1\\
&  &  & 1 & p_{2}\\
&  & \udots & \udots & \\
& 1 & p_{d-1} &  & \\
1 & {\prod\limits_{k=1}^{d-1}}\frac{1}{p_{k}} &  &  &
\end{array}
\right]  .
\]
Furthermore, given $\mathbf{m}=\left(  m_{1},m_{2},\cdots,m_{d-1}\right)
\in\mathbb{Z}^{d-1}$, we define the matrix-valued function:
\[
\mathbf{m}\mapsto D\left(  \mathbf{m}\right)  =\left[
\begin{array}
[c]{ccccc}%
\frac{1}{m_{1}} &  &  &  & \\
& \frac{1}{m_{2}} &  &  & \\
&  & \ddots &  & \\
&  &  & \frac{1}{m_{d-1}} & \\
&  &  &  & {\prod\limits_{k=1}^{d-1}}m_{k}%
\end{array}
\right]  .
\]

\begin{proposition}
\label{order3}Let $P,Q$ be unimodular integral matrices of order $d,$ and let
$U,V$ be unimodular integral matrices of order $2.$

\begin{enumerate}
\item If $p_{1},\cdots,p_{d-1}\neq0$ then $\left(  \left(  PM\left(
\mathbf{p}\right)  Q\right)  \mathbb{Z}^{d},\mathbb{Z}^{d}\right)  $ is a good
pair of lattices with common fundamental domain $PN\left(  \mathbf{p}\right)
\left[  0,1\right)  ^{d}.$

\item If $m_{1}m_{2}\cdots m_{d-1}\neq0$ then $\left(  \left(  PD\left(
\mathbf{m}\right)  Q\right)  \mathbb{Z}^{d},\mathbb{Z}^{d}\right)  $ is a good
pair of lattices.

\item If $m,n$ are nonzero integers such that $\gcd\left(  m,n\right)
=1\ $then
\[
\left(  \left(  U\left[
\begin{array}
[c]{cc}%
\frac{m}{n} & 0\\
0 & \frac{n}{m}%
\end{array}
\right]  V\right)  \mathbb{Z}^{2},\mathbb{Z}^{2}\right)
\]
is a good pair of lattices.
\end{enumerate}
\end{proposition}

It is worth mentioning that Part $3$ of Proposition \ref{order3} has also been
proved in \cite{Pfander}, Proposition $5.3.$ However, the novelty here lies in
our proof.

Next, for any real number $r,$ we define the matrix-valued function
\[
r\mapsto R\left(  r\right)  =%
\begin{bmatrix}
\sqrt{r} & 0\\
0 & \frac{1}{\sqrt{r}}%
\end{bmatrix}
.
\]

\begin{proposition}
\label{notgood} For any unimodular integral matrices $P,Q,$ if $r$ is a
natural number such that $\sqrt{r}$ is irrational then $\left(  \left(
PR\left(  r\right)  Q\right)  \mathbb{Z}^{2},\mathbb{Z}^{2}\right)  $ is not a
good pair of lattices.
\end{proposition}

We remark that Proposition \ref{notgood} is consistent with Proposition $5.3,$
\cite{Pfander} where it is proved that it is not possible to find a
star-shaped common fundamental domain for the lattices $R(2) \mathbb{Z}%
^{2},\mathbb{Z}^{2}$.

The present work is organized around the proofs of the results mentioned
above. In the second section we fix notations and present several results
crucial to the third section of the paper, in which we prove our main propositions.

\section{Generalities and Intermediate Results}

We remark that the investigation of good pairs of lattices in $\mathbb{R}^{d}$
where $d=1$ is not interesting. In fact, let us suppose that $\Gamma
_{1},\Gamma_{2}$ are two full-rank lattices of the same volume in
$\mathbb{R}.$ Then there exist nonzero real numbers $a,b$ such that
$\Gamma_{1}=a\mathbb{Z}\text{ and }\Gamma_{2}=b\mathbb{Z}$ and $\left\vert
a\right\vert =\left\vert b\right\vert .$ Thus, the half-open interval
$\left\vert a\right\vert \left[  0,1\right)  $ is a common fundamental domain
for the pair $\left(  \Gamma_{1},\Gamma_{2}\right)  ,$ and clearly $\left(
\Gamma_{1},\Gamma_{2}\right)  $ is a good pair of lattices. As such, in the
one-dimensional case every pair of lattices of the same volume is a good pair.
However, as we shall see in Proposition \ref{notgood}, there exist lattices of
the same volume in dimension two which are not good pairs.

\subsection{Notation and Terminology}

Throughout this paper, we shall assume that $d$ is a natural number strictly
greater than one. Let $M$ be a matrix. The transpose of the matrix $M$ is
denoted $M^{tr}.$ Let $v$ be a vector (in column form) in $\mathbb{R}^{d}.$
The Euclidean norm of $v$ is given by $\left\Vert v\right\Vert _{2}=\left(
\sum_{k=1}^{d}v_{k}^{2}\right)  ^{1/2},$ where%
\[
v=\left[
\begin{array}
[c]{ccc}%
v_{1} & \cdots & v_{d}%
\end{array}
\right]  ^{tr}.
\]
Given two vectors $v,w\in\mathbb{R}^{d},$ the inner product of $v$ and $w$ is
$\left\langle v,w\right\rangle =\sum_{k=1}^{d}v_{k}w_{k}.$

All subsets of $\mathbb{R}^{d}$ that we are concerned with in this paper will
be assumed to be Lebesgue measurable. Let $E$ be a subset of $\mathbb{R}^{d}.$
Then $\chi_{E}$ stands for the \textbf{indicator function} of the set $E.$
That is, $\chi_{E}:\mathbb{R}^{d}\rightarrow\mathbb{R}$ is the function
defined by
\[
\chi_{E}\left(  x\right)  =\left\{
\begin{array}
[c]{c}%
1\text{ if }x\in E\\
0\text{ if }x\notin E
\end{array}
\right.  .
\]
For any subset $E\subseteq$ $\mathbb{R}^{d}$ we define the set $E-E$ as
follows:
\[
E-E=\left\{  x-y\in\mathbb{R}^{d}:x,y\in E\right\}  .
\]
Throughout this paper, $\mathbf{0}$ stands for the zero vector in
$\mathbb{R}^{d},$ and we recall that $M$ is a \textbf{unimodular matrix} if
and only if $\det M=\pm1.$

\subsection{General Facts about Lattices and Good Pairs of Lattices}

\begin{lemma}
\label{First}Let $P,M$ be two matrices of the same order such that $\left\vert
\det P\right\vert =\left\vert \det M\right\vert .$ Then $\left(
P\mathbb{Z}^{d},M\mathbb{Z}^{d}\right)  $ is a good pair of lattices if and
only if for any invertible matrix $N$ of order $d,$ $\left(  \left(
NP\right)  \mathbb{Z}^{d},\left(  NM\right)  \mathbb{Z}^{d}\right)  $ is a
good pair of lattices.
\end{lemma}

\begin{proof}
Assume that $\left(  P\mathbb{Z}^{d},M\mathbb{Z}^{d}\right)  $ is a good pair
of lattices. Then from \cite{Pfander}, Page $3$ we know that $\left(
P\mathbb{Z}^{d},M\mathbb{Z}^{d}\right)  $ is a good pair of lattices if and
only if there exists a set $E=Z[0,1)^{d}$ such that $\sum_{k\in\mathbb{Z}^{d}%
}\chi_{E}\left(  x+Pk\right)  =\sum_{k\in\mathbb{Z}^{d}}\chi_{E}\left(
x+Mk\right)  =1$ for all $x\in\mathbb{R}^{d},$ where $Z$ is a matrix of order
$d$ and $\left\vert \det Z\right\vert =\left\vert \det P\right\vert
=\left\vert \det M\right\vert $. We shall show that the functions
\[
\mathbb{R}\ni x\mapsto\sum_{k\in\mathbb{Z}^{d}}\chi_{NE}\left(  x+NPk\right)
\text{ and }\mathbb{R}\ni x\mapsto\sum_{k\in\mathbb{Z}^{d}}\chi_{NE}\left(
x+NMk\right)
\]
are each equal to the constant function $\mathbb{R}\ni x\mapsto1.$\newline
Indeed, given any $x\in\mathbb{R}^{d},$ since $\sum_{k\in\mathbb{Z}^{d}}%
\chi_{E}\left(  x+Pk\right)  $ is equal to one, it follows that
\[
\sum_{k\in\mathbb{Z}^{d}}\chi_{NE}\left(  x+NPk\right)  =\sum_{k\in
\mathbb{Z}^{d}}\chi_{NE}\left(  NN^{-1}x+NPk\right)  =\sum_{k\in\mathbb{Z}%
^{d}}\chi_{E}\left(  N^{-1}x+Pk\right)  =1.
\]
Similarly, using the fact that $\sum_{k\in\mathbb{Z}^{d}}\chi_{E}\left(
x+Mk\right)  =1$ for all $x\in\mathbb{R}^{d},$ we obtain:
\[
\sum_{k\in\mathbb{Z}^{d}}\chi_{NE}\left(  x+NMk\right)  =\sum_{k\in
\mathbb{Z}^{d}}\chi_{NE}\left(  N\left(  N^{-1}x+Mk\right)  \right)
=\sum_{k\in\mathbb{Z}^{d}}\chi_{E}\left(  N^{-1}x+Mk\right)  =1.
\]
Therefore, $NE$ is a fundamental domain for $NP\mathbb{Z}^{d}$ and for
$NM\mathbb{Z}^{d}$ as well.

Now, let us assume that $\left(  NP\mathbb{Z}^{d},NM\mathbb{Z}^{d}\right)  $
is a good pair of lattices. That is, there is a set $E=Z[0,1)^{d}$ for some
matrix $Z$ such that
\[
\sum_{k\in\mathbb{Z}^{d}}\chi_{E}\left(  x+NPk\right)  =\sum_{k\in
\mathbb{Z}^{d}}\chi_{E}\left(  x+NMk\right)  =1
\]
$\text{ for all }x\in\mathbb{R}^{d}.$ Next,
\begin{align*}
\sum_{k\in\mathbb{Z}^{d}}\chi_{E}\left(  x+NPk\right)   &  =\sum
_{k\in\mathbb{Z}^{d}}\chi_{N^{-1}E}\left(  N^{-1}x+Pk\right)  =1\\
\sum_{k\in\mathbb{Z}^{d}}\chi_{E}\left(  x+NMk\right)   &  =\sum
_{k\in\mathbb{Z}^{d}}\chi_{N^{-1}E}\left(  N^{-1}x+Mk\right)  =1
\end{align*}
and $N^{-1}E$ is a common fundamental domain for $M\mathbb{Z}^{d}$ and
$P\mathbb{Z}^{d}.$
\end{proof}

\begin{lemma}
\label{integral}The following holds true:

\begin{enumerate}
\item Let $\Gamma=M\mathbb{Z}^{d}$ where $M$ is an invertible matrix with
entries in $\mathbb{Z}$. If $\left\vert \det M\right\vert =1$ then
$\Gamma=M\mathbb{Z}^{d}=\mathbb{Z}^{d}.$

\item Let $\Gamma_{1}=M_{1}\mathbb{Z}^{d},$ $\Gamma_{2}=M_{2}\mathbb{Z}^{d}$
be two full-rank lattices of the same volume. Then $\Gamma_{1}=\Gamma_{2}$ if
and only if $M_{1}=M_{2}U$ for some integral unimodular matrix $U.$
\end{enumerate}
\end{lemma}

\begin{proof}
For the first part, if $M$ is an integral matrix, then clearly, $\Gamma$ is a
subgroup of $\mathbb{Z}^{d}.$ In order to prove that $\mathbb{Z}^{d}$ is a
subgroup of $\Gamma,$ it is enough to show that the canonical basis elements
of the lattice $\mathbb{Z}^{d}$ are also elements of $M\mathbb{Z}^{d}.$ Let
$\left\{  e_{1},\cdots,e_{d}\right\}  $ be the canonical basis for the lattice
$\mathbb{Z}^{d}.$ That is, the matrix $\left[
\begin{array}
[c]{ccc}%
e_{1} & \cdots & e_{d}%
\end{array}
\right]  $ is the identity matrix of order $d$. Now, let $b_{j}=M^{-1}e_{j}$
for $j\in\left\{  1,\cdots,d\right\}  .$ Since $\left\vert \det M\right\vert
=1$ and $M^{-1}$ is an integral matrix, it is clear that each $b_{j}$ is an
integral vector and $Mb_{j}=e_{j}.$ Thus the set containing vectors
$e_{1},\cdots,e_{d}$ is a subset of $M\mathbb{Z}^{d}$ and $\mathbb{Z}^{d}$ is
a subgroup of $\Gamma$.

For the second part, assume that $\Gamma_{1}=\Gamma_{2}$. For each
$k\in\left\{  1,2,\cdots,d\right\}  $ there exists $\ell_{k}\in\mathbb{Z}^{d}$
such that $M_{1}e_{k}=M_{2}\ell_{k}.$ Next, let $U=\left[
\begin{array}
[c]{ccc}%
\ell_{1} & \cdots & \ell_{d}%
\end{array}
\right]  $ be a matrix of order $d$. By assumption, $M_{1}=M_{2}U.$ Moreover,
since $U=M_{2}^{-1}M_{1}$ then $\left\vert \det U\right\vert =1.$ Next, let us
suppose that $M_{1}=M_{2}U$ for some integral matrix $U$ where $\left\vert
\det U\right\vert =1.$ For any $\ell\in\mathbb{Z}^{d},$ $M_{1}\ell
=M_{2}\left(  U\ell\right)  \in M_{2}\mathbb{Z}^{d}.$ It follows that for any
$\ell\in\mathbb{Z}^{d},$ $M_{2}\ell=M_{1}\left(  U^{-1}\ell\right)  \in
M_{1}\mathbb{Z}^{d}.$ Therefore, $\Gamma_{1}=\Gamma_{2}.$ This completes the proof.
\end{proof}

\begin{proposition}
\label{changebasis} Let $M,M_{1},M_{2}$ be invertible matrices of order $d.$ Then

\begin{enumerate}
\item $\left(  M\mathbb{Z}^{d},\mathbb{Z}^{d}\right)  $ is a good pair of
lattices if and only if for any unimodular integral matrices $P$ and $T,$
$\left(  PMT\mathbb{Z}^{d},\mathbb{Z}^{d}\right)  $ is a good pair of lattices.

\item $\left(  M_{1}\mathbb{Z}^{d},M_{2}\mathbb{Z}^{d}\right)  $ is a good
pair of lattices if and only if $\left(  \mathbb{Z}^{d},M_{1}^{-1}%
M_{2}\mathbb{Z}^{d}\right)  $ is a good pair of lattices.

\item $\left(  M\mathbb{Z}^{d},\mathbb{Z}^{d}\right)  $ is a good pair of
lattices if and only if $\left(  M^{-1}\mathbb{Z}^{d},\mathbb{Z}^{d}\right)  $
is a good pair of lattices.
\end{enumerate}
\end{proposition}

\begin{proof}
For Part $1,$ assume that $\left(  M\mathbb{Z}^{d},\mathbb{Z}^{d}\right)  $ is
a good pair of lattices. Let $T,P$ be two unimodular integral matrices. Then
$\left(  M\mathbb{Z}^{d},\mathbb{Z}^{d}\right)  =\left(  M\left(
T\mathbb{Z}^{d}\right)  ,\mathbb{Z}^{d}\right)  .$ By applying Lemma
\ref{First} we see that $\left(  PMT\mathbb{Z}^{d},P\mathbb{Z}^{d}\right)  $
is a good pair of lattices. However, according to Lemma \ref{integral} Part
$1,$ we have $P\mathbb{Z}^{d}=\mathbb{Z}^{d}.$ Therefore, $\left(
PMT\mathbb{Z}^{d},\mathbb{Z}^{d}\right)  $ is a good pair of lattices. Now,
for the converse, let us assume that $\left(  PMT\mathbb{Z}^{d},\mathbb{Z}%
^{d}\right)  $ is a good pair of lattices. Since the inverse of $P$ is an
integral unimodular matrix, then
\[
\left(  MT\mathbb{Z}^{d},P^{-1}\mathbb{Z}^{d}\right)  =\left(  MT\mathbb{Z}%
^{d},\mathbb{Z}^{d}\right)  =\left(  M\mathbb{Z}^{d},\mathbb{Z}^{d}\right)
\]
is a good pair of lattices. This completes the proof of Part $1.$

Part $2$ follows from Lemma \ref{First}; indeed, $\left(  M_{1}\mathbb{Z}%
^{d},M_{2}\mathbb{Z}^{d}\right)  $ is a good pair of lattices if and only if
\[
\left(  M_{1}^{-1}\left(  M_{1}\mathbb{Z}^{d}\right)  ,M_{1}^{-1}\left(
M_{2}\mathbb{Z}^{d}\right)  \right)  =\left(  \mathbb{Z}^{d},\left(
M_{1}^{-1}M_{2}\right)  \mathbb{Z}^{d}\right)
\]
is a good pair of lattices. Similarly, Part $3$ follows from Lemma \ref{First}
as well and is simply due to the fact that
\[
\left(  M^{-1}M\mathbb{Z}^{d},M^{-1}\mathbb{Z}^{d}\right)  =\left(
\mathbb{Z}^{d},M^{-1}\mathbb{Z}^{d}\right)  .
\]

\end{proof}

The following lemmas play a central role in the proof of our main results.

\begin{lemma}
\label{common}Let $\Gamma_{1}=M\mathbb{Z}^{d}$ such that $\left\vert \det
M\right\vert =1.$ $E$ is a common fundamental domain for $\Gamma_{1}$ and
$\Gamma_{2}=\mathbb{Z}^{d}$ if and only if $\left(  E-E\right)  \cap
M\mathbb{Z}^{d}=\left\{  \mathbf{0}\right\}  ,\text{ }\left(  E-E\right)
\cap\mathbb{Z}^{d}=\left\{  \mathbf{0}\right\}  $ and the volume of $E$ is
equal to one.
\end{lemma}

\begin{proof}
Assume that $E$ is a common fundamental domain for $\Gamma_{1}$ and
$\Gamma_{2}=\mathbb{Z}^{d}.$ Then clearly, the volume of the set $E$ must be
equal to one. Next, given distinct $k,l\in\mathbb{Z}^{d},$ it is clear that
$\left(  E+Mk\right)  \cap\left(  E+Ml\right)  $ is an empty set. Therefore,
given any $x,y\in E,$ it must be true that $x-y$ can never be equal to $Mn$
for some $n\in\mathbb{Z}^{d}$ unless $n=\mathbf{0}.$ Therefore, $\left(
E-E\right)  \cap M\mathbb{Z}^{d}=\left\{  \mathbf{0}\right\}  .$ If $M$ is the
identity matrix, a similar argument allows us to derive that $\left(
E-E\right)  \cap\mathbb{Z}^{d}=\left\{  \mathbf{0}\right\}  $ as well.

Next, assuming that $\left(  E-E\right)  \cap M\mathbb{Z}^{d}=\left\{
\mathbf{0}\right\}  $ and $\left(  E-E\right)  \cap\mathbb{Z}^{d}=\left\{
\mathbf{0}\right\}  ,$ a calculation similar to that found in the proof of
Lemma~\ref{times} shows that $\left(  E+Mk\right)  \cap\left(  E+Ml\right)  $
is an empty set for $l$ not equal to $k$. Finally, since it is assumed that
the volume of $E$ is equal to one then $E$ is a common fundamental domain for
$\Gamma_{1}$ and $\Gamma_{2}=\mathbb{Z}^{d}.$ This completes the proof.
\end{proof}

\begin{lemma}
\label{good pair}Assume that $\left\vert \det M\right\vert =1.$ Then $\left(
M\mathbb{Z}^{d},\mathbb{Z}^{d}\right)  $ is a good pair if and only if there
exists a unimodular matrix $N$ such that $N\left(  -1,1\right)  ^{d}\cap
M\mathbb{Z}^{d}=\left\{  \mathbf{0}\right\}  ,\text{ and }N\left(
-1,1\right)  ^{d}\cap\mathbb{Z}^{d}=\left\{  \mathbf{0}\right\}  .$
\end{lemma}

\begin{proof}
$\left(  M\mathbb{Z}^{d},\mathbb{Z}^{d}\right)  $ is a good pair if and only
if there exists a common fundamental domain $E=N\left[  0,1\right)  ^{d}$ for
the lattices $M\mathbb{Z}^{d},\mathbb{Z}^{d}$ where $N$ is a unimodular
matrix. Now, appealing to Lemma \ref{common}, this holds if and only if
$\left(  E-E\right)  \cap M\mathbb{Z}^{d}=\left\{  \mathbf{0}\right\}  $ and
$\left(  E-E\right)  \cap\mathbb{Z}^{d}=\left\{  \mathbf{0}\right\}  .$
Finally, the proof is completed by observing that $\left(  E-E\right)
=N\left[  0,1\right)  ^{d}-N\left[  0,1\right)  ^{d}=N\left(  -1,1\right)
^{d}.$
\end{proof}

Appealing to Lemma \ref{common}, the following is immediate:

\begin{lemma}
\label{condition}Let $\Gamma_{1}=M\mathbb{Z}^{d}$ such that $\left\vert \det
M\right\vert =1.$ Then $M\left[  0,1\right)  ^{d}$ is a common fundamental
domain for $M\mathbb{Z}^{d}$ and $\mathbb{Z}^{d}$ if and only if $M\left(
-1,1\right)  ^{d}\cap\mathbb{Z}^{d}=\left\{  \mathbf{0}\right\}  .$
\end{lemma}

\subsection{Constructing Good Pairs from Known Good Pairs}

\begin{lemma}
\label{times}Let $\Gamma=\Gamma_{1}\times\Gamma_{2}$ where $\Gamma_{1}$ is a
full-rank lattice in $\mathbb{R}^{n}$ and $\Gamma_{2}$ is a full-rank lattice
in $\mathbb{R}^{m}.$ If $E_{1}$ is a common fundamental domain for $\Gamma
_{1}$ and $\mathbb{Z}^{n}$ in $\mathbb{R}^{n}$ and $E_{2}$ is a common
fundamental domain for $\Gamma_{2}$ and $\mathbb{Z}^{m}$ in $\mathbb{R}^{m}$
then $E=E_{1}\times E_{2}$ is a common fundamental domain for $\Gamma$ and
$\mathbb{Z}^{n}\times\mathbb{Z}^{m}$ in $\mathbb{R}^{n}\times\mathbb{R}^{m}.$
\end{lemma}

\begin{proof}
Indeed, let us assume that $E_{1}$ is a common fundamental domain for
$\Gamma_{1}$ and $\mathbb{Z}^{n}$ in $\mathbb{R}^{n},$ $E_{2}$ is a common
fundamental domain for $\Gamma_{2}$ and $\mathbb{Z}^{m}$ in $\mathbb{R}^{m},$
and there exist distinct $\gamma,k\in\mathbb{Z}^{n}\times\mathbb{Z}^{m}$ such
that the set $\left(  E+\gamma\right)  \cap\left(  E+k\right)  $ is not empty
($E=E_{1}\times E_{2}$). Then there exist $z,z^{\prime}\in E$ such that
$z+\gamma=z^{\prime}+k.$ Now, we write $z=\left(  x,y\right)  $, $z^{\prime
}=\left(  x^{\prime},y^{\prime}\right)  ,$ $\gamma=\left(  \gamma_{1}%
,\gamma_{2}\right)  $ and $k=\left(  k_{1},k_{2}\right)  .$ Thus,
\[
\left(  x,y\right)  +\gamma=\left(  x+\gamma_{1},y+\gamma_{2}\right)  =\left(
x^{\prime}+k_{1},y^{\prime}+k_{2}\right)  .
\]
As a result, $x+\gamma_{1}=x^{\prime}+k_{1}$ and $y+\gamma_{2}=y^{\prime
}+k_{2}$. Since $\gamma\neq k$ then either $\gamma_{1}\neq k_{1}$ or
$\gamma_{2}\neq k_{2}.$ So, we obtain that either $\gamma_{1}\neq k_{1}$ and
$x+\gamma_{1}=x^{\prime}+k_{1},$ or $\gamma_{2}\neq k_{2}$ and $y+\gamma
_{2}=y^{\prime}+k_{2}.$ This contradicts our assumption that $E_{1}$ is a
common fundamental domain for $\Gamma_{1}$ and $\mathbb{Z}^{n}$ in
$\mathbb{R}^{n}$ and $E_{2}$ is a common fundamental domain for $\Gamma_{2}$
and $\mathbb{Z}^{m}$ in $\mathbb{R}^{m}.$
\end{proof}

Appealing to Lemma \ref{times}, the following is immediate.

\begin{lemma}
\label{direct sum} Let $M_{1}$ and $M_{2}$ be two invertible matrices of order
$d.$ Assume that $\left(  M_{1}\mathbb{Z}^{n},\mathbb{Z}^{n}\right)  $ and
$\left(  M_{2}\mathbb{Z}^{m},\mathbb{Z}^{m}\right)  $ are good pairs of
lattices. If
\[
M=M_{1}\oplus M_{2}=\left[
\begin{array}
[c]{cc}%
M_{1} & \\
& M_{2}%
\end{array}
\right]
\]
then $\left(  M\left(  \mathbb{Z}^{n}\times\mathbb{Z}^{m}\right)
,\mathbb{Z}^{n}\times\mathbb{Z}^{m}\right)  $ is a good pair of lattices.
\end{lemma}

Given two matrices $A,B$ of order $a$ and $b$ respectively, such that
$A=\left[  A_{i,j}\right]  _{1\leq i,j\leq a}$ the \textbf{tensor product} (or
\textbf{Kronecker product}) of the matrices $A\otimes B$ is a matrix of order
$ab$ given by
\[
A\otimes B=\left[
\begin{array}
[c]{ccc}%
A_{11}B & \cdots & A_{1d}B\\
\vdots & \cdots & \vdots\\
A_{d1}B & \cdots & A_{dd}B
\end{array}
\right]  .
\]

\begin{lemma}
\label{tsor}Let $I_{p}$ be the identity matrix of order $p,$ and let $M$ be an
invertible matrix of order $d.$

\begin{enumerate}
\item If $\left(  M\mathbb{Z}^{d},\mathbb{Z}^{d}\right)  $ is a good pair of
lattices then $\left(  \left(  I_{p}\otimes M\right)  \mathbb{Z}%
^{pd},\mathbb{Z}^{pd}\right)  $ is a good pair of lattices.

\item If $\left(  M\mathbb{Z}^{d},\mathbb{Z}^{d}\right)  $ is a good pair of
lattices then $\left(  \left(  M\otimes I_{p}\right)  \mathbb{Z}%
^{pd},\mathbb{Z}^{pd}\right)  $ is a good pair of lattices.
\end{enumerate}
\end{lemma}

\begin{proof}
For Part $1$, we observe that
\[
I_{p}\otimes M=\left[
\begin{array}
[c]{ccc}%
M &  & \\
& \ddots & \\
&  & M
\end{array}
\right]  .
\]
Applying Lemma \ref{direct sum} an appropriate number of times gives us the
desired result. For Part $2,$ let
\[
M=\left[
\begin{array}
[c]{ccc}%
m_{11} & \cdots & m_{1d}\\
\vdots & \cdots & \vdots\\
m_{d1} & \cdots & m_{dd}%
\end{array}
\right]  \text{ and }I_{p}=\left[
\begin{array}
[c]{ccc}%
1 &  & \\
& \ddots & \\
&  & 1
\end{array}
\right]  .
\]
Then
\[
M\otimes I_{p}=\left[
\begin{array}
[c]{ccc}%
m_{11}I_{p} & \cdots & m_{1d}I_{p}\\
\vdots & \cdots & \vdots\\
m_{d1}I_{p} & \cdots & m_{dd}I_{p}%
\end{array}
\right]  .
\]
We shall show that $M\otimes I_{p}$ and $I_{p}\otimes M$ are similar matrices.
In other words, there exists an integral unimodular matrix $P$ such that:
\begin{equation}
P\left(  M\otimes I_{p}\right)  P^{-1}=\left(  I_{p}\otimes M\right)  .
\label{dirsum}%
\end{equation}
Indeed, let $\left\{  e_{i}\otimes e_{j}:1\leq i\leq d,1\leq j\leq p\right\}
$ be a basis for $\mathbb{R}^{d}\otimes\mathbb{R}^{p}.$ Define $Q:\mathbb{R}%
^{d}\otimes\mathbb{R}^{p}\rightarrow\mathbb{R}^{p}\otimes\mathbb{R}^{d}$ such
that $Q\left(  e_{i}\otimes e_{j}\right)  =e_{j}\otimes e_{i}.$ It is easy to
see that $Q$ is a linear isomorphism whose matrix is an integral unimodular
matrix. Moreover,
\begin{align*}
Q^{-1}\left(  M\otimes I_{p}\right)  Q\left(  e_{i}\otimes e_{j}\right)   &
=Q^{-1}\left(  M\otimes I_{p}\left(  e_{j}\otimes e_{i}\right)  \right) \\
&  =Q^{-1}\left(  Me_{j}\otimes e_{i}\right) \\
&  =e_{i}\otimes Me_{j}\\
&  =\left(  I_{p}\otimes M\right)  \left(  e_{i}\otimes e_{j}\right)  .
\end{align*}
Formula (\ref{dirsum}) is finally obtained by setting $Q=P^{-1}.$ Now, since
$\left(  \left(  I_{p}\otimes M\right)  \mathbb{Z}^{pd},\mathbb{Z}%
^{pd}\right)  $ is a good pair of lattices by Part $1$, it follows that
$\left(  \left(  M\otimes I_{p}\right)  \mathbb{Z}^{pd},\mathbb{Z}%
^{pd}\right)  $ is a good pair of lattices.
\end{proof}

\begin{proposition}
\label{tensor}Let $M$ be an invertible matrix of order $d.$ If $N$ is a
unimodular integral matrix of order $n$ and if $\left(  M\mathbb{Z}%
^{d},\mathbb{Z}^{d}\right)  $ is a good pair of lattices then $\left(  \left(
M\otimes N\right)  \mathbb{Z}^{dn},\mathbb{Z}^{dn}\right)  $ is a good pair of lattices.
\end{proposition}

\begin{proof}
We observe that $M\otimes N=\left(  M\otimes I_{n}\right)  \left(
I_{d}\otimes N\right)  .$ If $N$ is a unimodular integral matrix of order $n$
then $I_{d}\otimes N$ is a unimodular integral matrix of order $dn.$ Next,
since $\left(  M\mathbb{Z}^{d},\mathbb{Z}^{d}\right)  $ is a good pair of
lattices then it follows from Lemma \ref{tsor} Part $2,$ that $\left(  \left(
M\otimes I_{p}\right)  \mathbb{Z}^{pd},\mathbb{Z}^{pd}\right)  $ is a good
pair. Now, since $I_{d}\otimes N$ is a unimodular integral matrix, appealing
to Lemma \ref{integral} we obtain the desired result: $\left(  M\otimes
I_{p}\right)  \mathbb{Z}^{pd}=\left(  M\otimes N\right)  \mathbb{Z}^{pd}.$
\end{proof}

\section{Proofs of Main Results}

\subsection{Proof of Proposition \ref{Main}}

The fact that $\left(  \Gamma_{1},\Gamma_{2}\right)  $ is a good pair of
lattices if and only if $\left(  \left(  M_{1}^{-1}M_{2}\right)
\mathbb{Z}^{d},\mathbb{Z}^{d}\right)  $ is a good pair is due to Part $2$ of
Proposition \ref{changebasis}. The fact that $\left(  \Gamma_{1},\Gamma
_{2}\right)  $ is a good pair of lattices is equivalent to the statement that
there exists a unimodular matrix $N$ such that $N\left(  -1,1\right)  ^{d}%
\cap\left(  M_{1}^{-1}M_{2}\right)  \mathbb{Z}^{d}=\left\{  \mathbf{0}%
\right\}  ,\text{ and }N\left(  -1,1\right)  ^{d}\cap\mathbb{Z}^{d}=\left\{
\mathbf{0}\right\}  $ is due to Lemma \ref{good pair}.

\subsection{Proof of Proposition \ref{unipotent}}

It suffices to show that $M\left[  0,1\right)  ^{d}$ is a common fundamental
domain for $M\mathbb{Z}^{d}$ and $\mathbb{Z}^{d}.$ First, let us assume that
$M$ is an upper triangular unipotent matrix. We will offer a proof by
induction on $d.$ For the base case, let us assume that $d=2.$ We define
\[
M_{s}=\left[
\begin{array}
[c]{cc}%
1 & s\\
0 & 1
\end{array}
\right]
\]
$\text{ for some }s\in\mathbb{R}.$

The fact that $M_{s}\left[  0,1\right)  ^{2}$ is a fundamental domain for the
lattice $M_{s}\mathbb{Z}^{2}$ is obvious. Now, let $z=\left[
\begin{array}
[c]{cc}%
x & y
\end{array}
\right]  ^{tr}\in\left(  -1,1\right)  ^{2}$ such that $M_{s}z=k=\left[
\begin{array}
[c]{cc}%
k_{1} & k_{2}%
\end{array}
\right]  ^{tr}\in\mathbb{Z}^{2}.$ We would like to show that $k_{1}=k_{2}=0.$
First, we observe that
\[
z=M_{s}^{-1}k=\left[
\begin{array}
[c]{cc}%
k_{1}-sk_{2} & k_{2}%
\end{array}
\right]  ^{tr}\in\left(  -1,1\right)  ^{2}.
\]
This is only possible if $k_{2}=k_{1}=0.$ Therefore $M_{s}\left(  -1,1\right)
^{2}\cap\mathbb{Z}^{2}=\left\{  \mathbf{0}\right\}  $ and by Lemma
\ref{condition}, $M_{s}\left[  0,1\right)  ^{2}$ is a common fundamental
domain for $M_{s}\mathbb{Z}^{2}$ and $\mathbb{Z}^{2}.$

Now, let us suppose that for all $d\leq m-1\in\mathbb{N}$ we have that
$M\left[  0,1\right)  ^{d}$ is a common fundamental domain for $M\mathbb{Z}%
^{d}$ and $\mathbb{Z}^{d}$ whenever $M$ is a unipotent matrix. More precisely,
let
\[
M=\left[
\begin{array}
[c]{ccccc}%
1 & a_{1} & a_{2} & \cdots & a_{m-1}\\
& 1 & a_{m} & \cdots & a_{2m-3}\\
&  & \ddots & \ddots & \vdots\\
&  &  & 1 & a_{\frac{_{m^{2}-m}}{2}}\\
&  &  &  & 1
\end{array}
\right]  \text{ }%
\]
be an arbitrary unipotent matrix of order $m$ with real entries. Let
\[
v=\left[
\begin{array}
[c]{cccc}%
a_{1} & a_{2} & \cdots & a_{m-1}%
\end{array}
\right]  ,\text{ }M_{1}=\left[
\begin{array}
[c]{cccc}%
1 & a_{m} & \cdots & a_{2m-3}\\
& \ddots & \ddots & \vdots\\
&  & 1 & a_{\frac{_{m^{2}-m}}{2}}\\
&  &  & 1
\end{array}
\right]
\]
so that
\[
M=\left[
\begin{array}
[c]{cc}%
1 & v\\
0 & M_{1}%
\end{array}
\right]  .
\]
Next, assume that for any given $z\in\left(  -1,1\right)  ^{m}$ we have that
$Mz\in\mathbb{Z}^{m}.$ We want to show that $z$ is the zero vector. Writing
\[
Mz=\left[
\begin{array}
[c]{cc}%
1 & v\\
0 & M_{1}%
\end{array}
\right]  \left[
\begin{array}
[c]{c}%
z_{1}\\
z_{2}%
\end{array}
\right]  =\left[
\begin{array}
[c]{c}%
z_{1}+\left\langle v,z_{2}\right\rangle \\
M_{1}z_{2}%
\end{array}
\right]  ,
\]
where $\left\langle v,z_{2}\right\rangle $ is the dot product of the vectors
$v,z_{2}$, it follows that $M_{1}z_{2}\in\mathbb{Z}^{m-1}.$ By the assumption
of the induction, then $z_{2}=0$ and it follows that $z_{1}+\left\langle
v,z_{2}\right\rangle =z_{1}\in\mathbb{Z}.$ Since $z\in\left(  -1,1\right)
^{m}$ then $z_{1}=0$ and this completes the induction.

Now, let us suppose that $M$ is a lower triangular unipotent matrix. Put
\[
J=\left[
\begin{array}
[c]{ccc}
&  & 1\\
& \udots & \\
1 &  &
\end{array}
\right]  .
\]
Notice that $JMJ^{-1}$ is an upper triangular matrix. Since $\left(
JMJ^{-1}\mathbb{Z}^{d},\mathbb{Z}^{d}\right)  $ is a good pair of lattices,
using the fact that $J$ is a unimodular integral matrix together with
Proposition \ref{changebasis}, Part $1,$ it follows that $\left(
M\mathbb{Z}^{d},\mathbb{Z}^{d}\right)  $ is a good pair of lattices as well.
This completes the proof of the first part.

\subsection{Proof of Proposition \ref{order3}}

Put
\[
M=M\left(  \mathbf{p}\right)  =\left[
\begin{array}
[c]{cccc}%
p_{1} & 1 &  & \\
& \ddots & \ddots & \\
&  & p_{d-1} & 1\\
&  &  & \frac{1}{p_{1}\cdots p_{d-1}}%
\end{array}
\right]  ,
\]
and
\[
N=N\left(  \mathbf{p}\right)  =\left[
\begin{array}
[c]{ccccc}
&  &  &  & 1\\
&  &  & 1 & p_{2}\\
&  & \udots & \udots & \\
& 1 & p_{d-1} &  & \\
1 & \frac{1}{p_{1}\cdots p_{d-1}} &  &  &
\end{array}
\right]  .
\]
We would like to show that $N\left[  0,1\right)  ^{d}$ is a common fundamental
domain for the pair $\left(  M\mathbb{Z}^{d},\mathbb{Z}^{d}\right)  .$ For
this purpose, it is enough to show (see Lemma \ref{common}) that $N\left(
-1,1\right)  ^{d}\cap\mathbb{Z}^{d}=\left\{  \mathbf{0}\right\}  $ and
\[
N\left(  -1,1\right)  ^{d}\cap M\mathbb{Z}^{d}=\left\{  \mathbf{0}\right\}  .
\]
In order to prove that $N\left(  -1,1\right)  ^{d}\cap M\mathbb{Z}%
^{d}=\left\{  \mathbf{0}\right\}  ,$ let us suppose that\ $Nv=Mk$,
$v\in\left(  -1,1\right)  ^{d}$ and $k\in\mathbb{Z}^{d}.$ It follows that
$M^{-1}Nv\in\mathbb{Z}^{d}.$ Computing the inverse of $M,$ we obtain
\begin{equation}
M^{-1}=\left[
\begin{array}
[c]{cccccc}%
\frac{1}{p_{1}} & \dfrac{\left(  -1\right)  ^{1}}{p_{1}p_{2}} & \dfrac{\left(
-1\right)  ^{2}}{p_{1}p_{2}p_{3}} & \cdots & \dfrac{\left(  -1\right)  ^{d-2}%
}{p_{1}p_{2}\cdots p_{d-1}} & \left(  -1\right)  ^{d-1}\\
& \dfrac{1}{p_{2}} & \dfrac{\left(  -1\right)  ^{1}}{p_{2}p_{3}} & \cdots &
\dfrac{\left(  -1\right)  ^{d-3}}{p_{2}\cdots p_{d-1}} & \left(  -1\right)
^{d-2}p_{1}\\
&  & \dfrac{1}{p_{3}} & \cdots & \dfrac{\left(  -1\right)  ^{d-2}}{p_{3}\cdots
p_{d-1}} & \left(  -1\right)  ^{d-3}p_{1}p_{2}\\
&  &  & \ddots & \vdots & \vdots\\
&  &  &  & \dfrac{1}{p_{d-1}} & \left(  -1\right)  ^{d-\left(  d-1\right)
}p_{1}\cdots p_{d-2}\\
&  &  &  &  & p_{1}\cdots p_{d-1}%
\end{array}
\right]  .
\end{equation}
Next, with some formal calculations we obtain that
\[
M^{-1}Nv=\left[
\begin{array}
[c]{cccccc}%
\left(  -1\right)  ^{d-1} &  &  &  &  & \\
\left(  -1\right)  ^{d-2}p_{1} &  &  &  &  & 1\\
\left(  -1\right)  ^{d-3}p_{1}p_{2} &  &  &  & \udots & \\
\vdots &  &  & 1 &  & \\
\left(  -1\right)  ^{d-\left(  d-1\right)  }p_{1}\cdots p_{d-2} &  & 1 &  &  &
\\
p_{1}\cdots p_{d-1} & 1 &  &  &  &
\end{array}
\right]  v=k\in\mathbb{Z}^{d}.
\]
Therefore, $v$ must be the zero vector.

To show that $N\left(  -1,1\right)  ^{d}\cap\mathbb{Z}^{d}=\left\{
\mathbf{0}\right\}  ,$ let $z\in\left(  -1,1\right)  ^{d}$ such that
$Nz=k\in\mathbb{Z}^{d}.$ More precisely, we have
\[
\left[
\begin{array}
[c]{ccccc}
&  &  &  & 1\\
&  &  & 1 & p_{2}\\
&  & \udots & \udots & \\
& 1 & p_{d-1} &  & \\
1 & \frac{1}{p_{1}\cdots p_{d-1}} &  &  &
\end{array}
\right]  \left[
\begin{array}
[c]{c}%
z_{1}\\
z_{2}\\
\vdots\\
z_{d-1}\\
z_{d}%
\end{array}
\right]  =\left[
\begin{array}
[c]{c}%
z_{d}\\
z_{d-1}+z_{d}\\
\vdots\\
z_{2}+z_{3}\\
z_{1}+\frac{z_{2}}{p_{1}\cdots p_{d-1}}%
\end{array}
\right]  =\left[
\begin{array}
[c]{c}%
k_{1}\\
k_{2}\\
\vdots\\
k_{d-1}\\
k_{d}%
\end{array}
\right]  .
\]
Now using the fact that $k\in\mathbb{Z}^{d}$ together with $z\in\left(
-1,1\right)  ^{d}$ gives us that $z$ must be equal to the zero vector.
Therefore, $N\left(  -1,1\right)  ^{d}\cap\mathbb{Z}^{d}=\left\{
\mathbf{0}\right\}  .$ In light of Proposition \ref{changebasis} Part $1$,
$\left(  \left(  PM\left(  \mathbf{p}\right)  Q\right)  \mathbb{Z}%
^{d},\mathbb{Z}^{d}\right)  $ is a good pair of lattices with common
fundamental domain $\left(  PN\left(  \mathbf{p}\right)  \right)  \left[
0,1\right)  ^{d}$ whenever $P,Q$ are integral unimodular matrices.

For Part $2,$ appealing again to Proposition \ref{changebasis} Part $1$ it is
enough to show that $\left(  D\left(  \mathbf{m}\right)  \mathbb{Z}%
^{d},\mathbb{Z}^{d}\right)  $ is a good pair of lattices. First, let
\[
Z=\left[
\begin{array}
[c]{ccccc}%
\dfrac{1}{m_{1}} & 1 &  &  & \\
& \dfrac{1}{m_{2}} & 1 &  & \\
&  & \ddots & \ddots & \\
&  &  & \dfrac{1}{m_{d-1}} & 1\\
&  &  &  & {\displaystyle\prod\limits_{i=1}^{d-1}}m_{i}%
\end{array}
\right]  .
\]
Next, applying the first part of the proposition it is clear that $\left(
Z\mathbb{Z}^{d},\mathbb{Z}^{d}\right)  $ is a good pair of lattices. Now, put
\[
U=\left[
\begin{array}
[c]{cccccc}%
1 & \left(  -1\right)  ^{1}m_{1} & \left(  -1\right)  ^{2}m_{1}m_{2} & \left(
-1\right)  ^{3}m_{1}m_{2}m_{3} & \cdots & \left(  -1\right)  ^{d-1}%
{\displaystyle\prod\limits_{i=1}^{d-1}}m_{i}\\
& 1 & \left(  -1\right)  ^{1}m_{2} & \left(  -1\right)  ^{2}m_{2}m_{3} &
\cdots & \left(  -1\right)  ^{d-2}{\displaystyle\prod\limits_{i=2}^{d-1}}%
m_{i}\\
&  & 1 & \left(  -1\right)  ^{1}m_{3} & \cdots & \left(  -1\right)
^{d-3}{\displaystyle\prod\limits_{i=3}^{d-1}}m_{i}\\
&  &  & \ddots & \ddots & \vdots\\
&  &  &  & 1 & \left(  -1\right)  ^{1}m_{d-1}\\
&  &  &  &  & 1
\end{array}
\right]  .
\]
Since $U$ is a unimodular integral matrix, then $Z\mathbb{Z}^{d}%
=ZU\mathbb{Z}^{d}$ (Lemma \ref{integral}). It is easy to check that $ZU$ is
equal to the diagonal matrix
\[
\left[
\begin{array}
[c]{ccccc}%
\dfrac{1}{m_{1}} &  &  &  & \\
& \dfrac{1}{m_{2}} &  &  & \\
&  & \ddots &  & \\
&  &  & \dfrac{1}{m_{d-1}} & \\
&  &  &  & {\displaystyle\prod\limits_{i=1}^{d-1}}m_{i}%
\end{array}
\right]  =D\left(  \mathbf{m}\right)  .
\]
Therefore $\left(  ZU\mathbb{Z}^{d},\mathbb{Z}^{d}\right)  =\left(  D\left(
\mathbf{m}\right)  \mathbb{Z}^{d},\mathbb{Z}^{d}\right)  $ is a good pair of
lattices. For the last part of the proposition, put
\[
S=\left[
\begin{array}
[c]{cc}%
0 & 1\\
1 & \frac{n}{m}%
\end{array}
\right]  ,S^{\prime}=\left[
\begin{array}
[c]{cc}%
\frac{m}{n} & 1\\
0 & \frac{n}{m}%
\end{array}
\right]  .
\]
We claim that $S\left[  0,1\right)  ^{2}$ is a common fundamental domain for
the lattices $S^{\prime}\mathbb{Z}^{2}$ and $\mathbb{Z}^{2}.$ To see this, it
suffices (see Proposition \ref{Main}) to check that $S\left(  -1,1\right)
^{2}\cap\mathbb{Z}^{2}=\left\{  \mathbf{0}\right\}  $ and $S\left(
-1,1\right)  ^{2}\cap S^{\prime}\mathbb{Z}^{2}=\left\{  \mathbf{0}\right\}  .$
Let $z=\left[
\begin{array}
[c]{cc}%
x & y
\end{array}
\right]  ^{tr}\in\left(  -1,1\right)  ^{2}$ such that $Sz\in\mathbb{Z}^{2}.$
Then $z=S^{-1}k=\left[
\begin{array}
[c]{cc}%
k_{2}-\frac{nk_{1}}{m} & k_{1}%
\end{array}
\right]  ^{tr}$ where $k=\left[
\begin{array}
[c]{cc}%
k_{1} & k_{2}%
\end{array}
\right]  ^{tr}\in\mathbb{Z}^{2}.$ So, $z=\mathbf{0}.$ Next, let us assume that
$Sz\in S^{\prime}\mathbb{Z}^{2}.$ That is, $Sz=S^{\prime}k$ for some
$k=\left[
\begin{array}
[c]{cc}%
k_{1} & k_{2}%
\end{array}
\right]  ^{tr}\in\mathbb{Z}^{2}.$ It follows that
\[
z=\left[
\begin{array}
[c]{cc}%
x & y
\end{array}
\right]  ^{tr}=S^{-1}S^{\prime}\left[
\begin{array}
[c]{cc}%
k_{1} & k_{2}%
\end{array}
\right]  ^{tr}=\left[
\begin{array}
[c]{cc}%
-k_{1} & \frac{mk_{1}+nk_{2}}{n}%
\end{array}
\right]  ^{tr}.
\]
Thus $z=\mathbf{0}$ as well. \newline Next, since $\gcd\left(  m,n\right)  =1$
there exist $\ell_{1},\ell_{2}\in\mathbb{Z}$ such that $1-n\ell_{2}-m\ell
_{1}=0.$ As such, it follows that
\begin{align*}
\left[
\begin{array}
[c]{cc}%
1 & -m\ell_{2}\\
0 & 1
\end{array}
\right]  \left[
\begin{array}
[c]{cc}%
\frac{m}{n} & 1\\
0 & \frac{n}{m}%
\end{array}
\right]  \left[
\begin{array}
[c]{cc}%
1 & -\ell_{1}n\\
0 & 1
\end{array}
\right]   &  =\left[
\begin{array}
[c]{cc}%
\frac{m}{n} & 1-n\ell_{2}-m\ell_{1}\\
0 & \frac{1}{m}n
\end{array}
\right] \\
&  =\left[
\begin{array}
[c]{cc}%
\frac{m}{n} & 0\\
0 & \frac{n}{m}%
\end{array}
\right]  .
\end{align*}
By Proposition \ref{changebasis}, Part $1,$%
\begin{equation}
\left(  \left[
\begin{array}
[c]{cc}%
\frac{m}{n} & 0\\
0 & \frac{n}{m}%
\end{array}
\right]  \mathbb{Z}^{2},\mathbb{Z}^{2}\right)  . \label{gp}%
\end{equation}
Using the fact that
\[
\left[
\begin{array}
[c]{cc}%
1 & -\ell_{1}n\\
0 & 1
\end{array}
\right]  \mathbb{Z}^{2}=\mathbb{Z}^{2}%
\]
together with Lemma \ref{First}, we conclude that
\[
\left[
\begin{array}
[c]{cc}%
1 & -m\ell_{2}\\
0 & 1
\end{array}
\right]  \left[
\begin{array}
[c]{cc}%
0 & 1\\
1 & \frac{n}{m}%
\end{array}
\right]  \left[  0,1\right)  ^{2}=\left[
\begin{array}
[c]{cc}%
-m\ell_{2} & 1-n\ell_{2}\\
1 & \frac{n}{m}%
\end{array}
\right]  \left[  0,1\right)  ^{2}%
\]
is a common connected fundamental domain for the pair (\ref{gp}). Finally, the
first part of Proposition \ref{changebasis} gives the desired result.

\subsection{Proof of Proposition \ref{notgood}}

According to Proposition \ref{changebasis}, it is enough to show that if $r$
is a natural number such that $\sqrt{r}$ is irrational then $\left(  R\left(
r\right)  \mathbb{Z}^{2},\mathbb{Z}^{2}\right)  $ is not a good pair of
lattices. Put $N=\left[
\begin{array}
[c]{cc}%
a & b\\
c & d
\end{array}
\right]  $ such that $\left\vert \det N\right\vert =1.$ Let us suppose that
$\Omega=N\left[  0,1\right)  ^{2}$ is a common fundamental domain for
$\mathbb{Z}^{2}$ and $R\mathbb{Z}^{2}=R\left(  r\right)  \mathbb{Z}^{2}.$
There must exist a non-zero element $k$ in $\mathbb{Z}^{2}$ such that one
corner of the closure of the set $\Omega+k$ meet $\Omega$ at the origin.
Similarly, since $\Omega$ tiles the plane by $R(r)\mathbb{Z}^{2},$ there must
exist a non-trivial element $k$ of $R(r)\mathbb{Z}^{2}$ such that one corner
of the closure of $\Omega+k$ intersects $\Omega$ at the origin as well (see
Figure below)

\begin{center}
\begin{figure}[ptbh]
\includegraphics[scale=0.5]{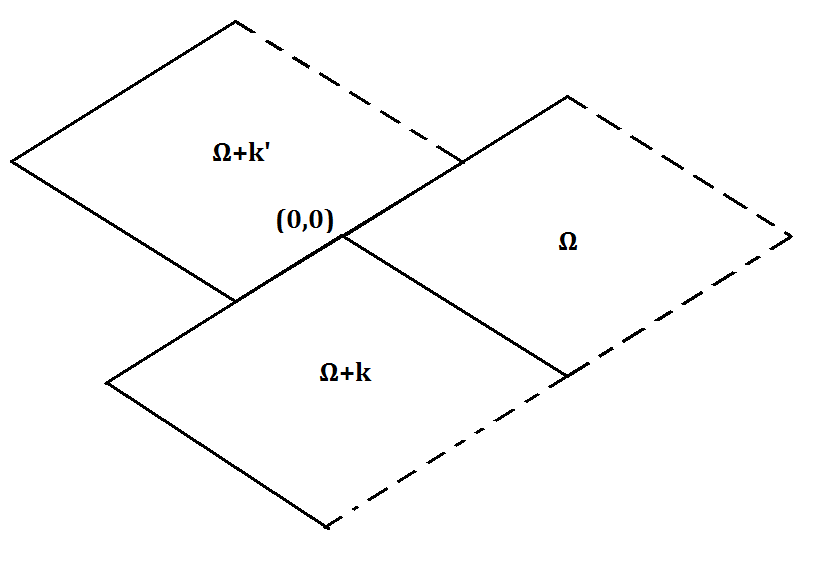}\caption{Behavior of a tiling
around the origin of the plane}%
\end{figure}
\end{center}

Hence, there exist
\[
p,q\in\left\{  \left[
\begin{array}
[c]{c}%
1\\
0
\end{array}
\right]  ,\left[
\begin{array}
[c]{c}%
0\\
1
\end{array}
\right]  \right\}
\]
such that
\begin{equation}
\left\{
\begin{array}
[c]{c}%
Np+k=\mathbf{0}\\
Nq+j=\mathbf{0}%
\end{array}
\right.  \label{system}%
\end{equation}
for some $k\in\mathbb{Z}^{2}-\left\{  \mathbf{0}\right\}  $ and $j\in
R\mathbb{Z} ^{2}-\left\{  \mathbf{0}\right\} .$ By assumption, $N$ is a
unimodular matrix. However, without loss of generality, we may assume that
$\det N=1$. Indeed, if $\det N=-1,$ then (\ref{system}) is equivalent to
$\left\{
\begin{array}
[c]{c}%
JNp+Jk=\mathbf{0}\\
JNq+Jj=\mathbf{0}%
\end{array}
\right.  $ where $\det\left(  JN\right)  =1,$ $Jk\in\mathbb{Z}^{2}-\left\{
\mathbf{0}\right\}  ,Jj\in R\mathbb{Z}^{2}-\left\{  \mathbf{0}\right\}  $ and
$J=\left[
\begin{array}
[c]{cc}%
-1 & 0\\
0 & 1
\end{array}
\right]  .$

We shall prove that if $\sqrt{r}\notin\mathbb{Q}$ then (\ref{system}) has no
solution. There are several possible cases that may arise from all the
possible choices for $p,q$. First of all, since $R\mathbb{Z}^{2}\cap
\mathbb{Z}^{2}=\left\{  \mathbf{0}\right\}  $, it is easy to see that
(\ref{system}) has no solution whenever $p=q.$ Therefore, we should only focus
on the cases where $p$ is not equal to $q.$ Put
\[
k=\left[
\begin{array}
[c]{cc}%
k_{1} & k_{2}%
\end{array}
\right]  ^{tr}\text{ and }j=\left[
\begin{array}
[c]{cc}%
\sqrt{r}j_{1} & \frac{1}{\sqrt{r}}j_{2}%
\end{array}
\right]  ^{tr}\text{ where }k_{1,}k_{2},j_{1},j_{2}\in\mathbb{Z}.
\]
\textbf{Case $1.1$} If
\[
p=\left[
\begin{array}
[c]{c}%
1\\
0
\end{array}
\right]  ,q=\left[
\begin{array}
[c]{c}%
0\\
1
\end{array}
\right]  ,N=\left[
\begin{array}
[c]{cc}%
\frac{bc+1}{d} & b\\
c & d
\end{array}
\right]
\]
and $d\neq0$ then
\[
k_{1}=\frac{j_{1}k_{2}r+\sqrt{r}}{j_{2}},k_{2}=-c,j_{1}=\frac{-b}{\sqrt{r}%
},j_{2}=-d\sqrt{r}%
\]
and $k_{1}j_{2}-j_{1}k_{2}r=\sqrt{r}.$ Thus, System (\ref{system}) has no
solution since $\sqrt{r}$ is irrational. \newline\textbf{Case $1.2$} If
\[
p=\left[
\begin{array}
[c]{c}%
1\\
0
\end{array}
\right]  ,q=\left[
\begin{array}
[c]{c}%
0\\
1
\end{array}
\right]  ,N=\left[
\begin{array}
[c]{cc}%
a & b\\
-\frac{1}{b} & 0
\end{array}
\right]
\]
and $b\neq0$ then
\[
k_{1}=a,k_{2}=-\frac{1}{b},j_{1}=\frac{b}{\sqrt{r}},j_{2}=0
\]
and $j_{1}=-\frac{1}{k_{2}\sqrt{r}},k_{2}\neq0.$ This is absurd since $j_{1}$
is an integer. \newline\textbf{Case $2.1$} If
\[
p=\left[
\begin{array}
[c]{c}%
0\\
1
\end{array}
\right]  ,q=\left[
\begin{array}
[c]{c}%
1\\
0
\end{array}
\right]  ,N=\left[
\begin{array}
[c]{cc}%
\frac{bc+1}{d} & b\\
c & d
\end{array}
\right]  ,d\neq0
\]
then
\[
k_{1}=-b,k_{2}=-d,j_{1}=\frac{-1-bc}{d\sqrt{r}},j_{2}=-c\sqrt{r}%
\]
and $k_{2}j_{1}r-j_{2}k_{1}=\sqrt{r}$ which is absurd. \newline\textbf{Case
$2.2$} If
\[
p=\left[
\begin{array}
[c]{c}%
0\\
1
\end{array}
\right]  ,q=\left[
\begin{array}
[c]{c}%
1\\
0
\end{array}
\right]  ,N=\left[
\begin{array}
[c]{cc}%
a & b\\
-\frac{1}{b} & 0
\end{array}
\right]  ,b\neq0
\]
then
\[
k_{1}=b,k_{2}=0,j_{1}=\frac{a}{\sqrt{r}},j_{2}=-\frac{\sqrt{r}}{b}.
\]
Therefore, $j_{2}=-\frac{\sqrt{r}}{k_{1}}$ and this is absurd.

\end{document}